\documentclass[a4paper,11pt,reqno]{amsart}
\usepackage{amsmath}
\usepackage{amstext}
\usepackage{amsbsy}
\usepackage{amsopn}
\usepackage{upref}
\usepackage{amsthm}
\usepackage{amsfonts}
\usepackage{amssymb}
\usepackage{mathrsfs}
\usepackage{times}
\allowdisplaybreaks
 \newtheorem{theorem}{Theorem}[section]
 
 \newtheorem{proposition}[theorem]{Proposition}
 \newtheorem{lemma}[theorem]{Lemma}
\theoremstyle{definition}
 
\theoremstyle{remark}
 
\begin{document}
%%%%%%
%%%%%% Make Title
%%%%%%
\title[Bargmann-type Transforms]{Bargmann-type transforms and modified harmonic oscillators}
\author[H.~Chihara]{Hiroyuki Chihara}
\address{College of Education, University of the Ryukyus, Nishihara, Okinawa 903-0213, Japan}
\email{aji@rencho.me}
\thanks{Supported by the JSPS Grant-in-Aid for Scientific Research \#19K03569.}
\begin{abstract}
We study some complete orthonormal systems on the real-line. 
These systems are determined by Bargmann-type transforms, 
which are Fourier integral operators with complex-valued quadratic phase functions. 
Each system consists of eigenfunctions for a second-order elliptic differential operator like the Hamiltonian of the harmonic oscillator. 
We also study the commutative case of a certain class of systems of second-order differential operators called the non-commutative harmonic oscillators. By using the diagonalization technique, we compute the eigenvalues and eigenfunctions for the commutative case of the non-commutative harmonic oscillators. 
Finally, we study a family of functions associated with an ellipse in the phase plane.  
We show that the family is a complete orthogonal system on the real-line.  
\end{abstract}
\maketitle
%%%%%%
%%%%%% section 1
%%%%%%
\section{Introduction}
\label{section:introduction}
We are concerned with the Bargmann-type transform defined by 
\begin{equation}
T_hu(z)
=
C_\phi
h^{-3/4}
\int_\mathbb{R}
e^{i\phi(z,x)/h}
u(x)
dx, 
\quad
z\in\mathbb{C},  
\label{equation:bargmann}
\end{equation}
where 
$\phi(z,x)$ 
is a complex-valued quadratic phase function of the form 
$$
\phi(z,x)
=
\frac{A}{2}z^2
+
Bzx
+
\frac{C}{2}x^2, 
\quad
A,B,C \in \mathbb{C} 
$$
with assumptions $B\ne0$ and $\operatorname{Im}C>0$, 
and 
$C_\phi=2^{-1/2}\pi^{-3/4}\lvert{B}\rvert(\operatorname{Im}C)^{-1/4}$. 
Throughout of the present paper, 
we deal with only the one-dimensional case for the sake of simplicity. 
It is possible to discuss higher dimensional case, and we omit the detail. 
Note that the integral transform \eqref{equation:bargmann} 
is well-defined for tempered distributions on $\mathbb{R}$ 
since 
$\operatorname{Re}\bigl(i\phi(z,x)\bigr)=\mathcal{O}(-\operatorname{Im} C x^2/3)$ 
for $\lvert{x}\rvert \rightarrow \infty$. 
\par
These integral transforms were introduced by Sj\"ostrand (See, e.g., \cite{sjoestrand}). 
He developed microlocal analysis based on them. 
One can see \eqref{equation:bargmann} 
as a global Fourier integral operator 
associated with a linear canonical transform 
$\kappa_T:\mathbb{C}^2\ni\bigl(x,-\phi^\prime_x(z,x)\bigr)\mapsto\bigl(z,\phi^\prime_z(z,x)\bigr)\in\mathbb{C}^2$, 
that is, 
\begin{equation}
\kappa_T:
\mathbb{C}^2
\ni
(x,\xi) 
\mapsto 
\left(-\frac{Cx+\xi}{B},Bx-\frac{A(Cx+\xi)}{B}\right)
\in
\mathbb{C}^2.  
\label{equation:kappa}
\end{equation}
If we set 
$\Phi(z)=\displaystyle\max_{x\in\mathbb{R}}\operatorname{Re}\bigl(i\phi(z,x)\bigr)$, 
then we have 
\begin{equation}
\Phi(z)
=
\frac{\lvert{Bz}\rvert^2}{4\operatorname{Im}C}
-
\operatorname{Re}
\left\{
\frac{(Bz)^2}{4\operatorname{Im}C}
+
\frac{Az^2}{2i}
\right\},
\label{equation:Phi} 
\end{equation}
$$
\kappa_T(\mathbb{R}^2)
=
\left\{
\left(z,\frac{2}{i}\frac{\partial \Phi}{\partial z}(z)\right)
\ \bigg\vert \ 
z\in\mathbb{C}
\right\}. 
$$
This means that the singularities of a tempered distribution $u$ 
described in the phase plane $\mathbb{R}^2$ are translated into those of $T_hu$ 
in the I-Lagrangian submanifold $\kappa_T(\mathbb{R}^2)$. 
The microlocal analysis of Sj\"ostrand is based on the equivalence of 
the Weyl quantization on $\mathbb{R}$, 
the Weyl quantization on $\kappa_T(\mathbb{R})$, 
and the Berezin-Toeplitz quantization on $\mathbb{C}$. 
For more detail about them, see \cite{sjoestrand} or \cite{chihara1}. 
\par
Let $L^2(\mathbb{R})$ be the set of all square-integrable functions on $\mathbb{R}$, 
and let $L^2_\Phi(\mathbb{C})$ be the set of all square-integrable functions 
on $\mathbb{C}$ with respect to a weighted measure $e^{-2\Phi(z)/h}L(dz)$, 
where $L$ is the Lebesgue measure on $\mathbb{C}\simeq\mathbb{R}^2$. 
We denote by $\mathscr{H}_\Phi(\mathbb{C})$ the set of all entire functions 
in $L^2_\Phi(\mathbb{C})$. 
It is well-known that $T_h$ gives a Hilbert space isomorphism of $L^2(\mathbb{R})$ 
onto $\mathscr{H}_\Phi(\mathbb{C})$, that is, 
$$
(T_hu,T_hv)_{L^2_\Phi}
=
(u,v)_{L^2} 
\quad
u,v \in L^2(\mathbb{R}), 
$$
where 
$$
(U,V)_{L^2_\Phi}
=
\int_\mathbb{C}
U(z)\overline{V(z)}
e^{-2\Phi(z)/h}
L(dz), 
\quad
U,V \in L^2_\Phi(\mathbb{C}), 
$$
$$
(u,v)_{L^2}
=
\int_{\mathbb{R}}
u(x)\overline{v(x)}
dx
\quad
u,v \in L^2(\mathbb{R}).
$$
We sometimes denote $(U,V)_{L^2_\Phi}$ for $U,V \in \mathscr{H}_\Phi(\mathbb{C})$ by 
$(U,V)_{\mathscr{H}_\Phi}$. 
The inverse mapping $T_h^\ast$ is given by 
$$
T_h^\ast U(x)
=
C_\phi
\int_\mathbb{C}
e^{-i \overline{\phi(z,x)}/h} 
U(z)e^{-2\Phi(z)/h} 
L(dz), 
\quad
x\in\mathbb{R}, 
\quad
U \in \mathscr{H}_\Phi(\mathbb{C}).
$$
Note that $T_h^\ast$ is well-defined for $U \in L^2_\Phi(\mathbb{C})$. 
$T_h{\circ}T_h^\ast$ becomes an orthogonal projector of 
$L^2_\Phi(\mathbb{C})$ onto $\mathscr{H}_\Phi(\mathbb{C})$. 
More concretely, 
\begin{equation}
T_h{\circ}T_h^\ast
U(z)
=
\frac{C_\Phi}{h}
\int_\mathbb{C}
e^{2\Psi(z,\bar{\zeta})/h}
U(\zeta)
e^{-2\Phi(\zeta)/h}
L(d\zeta), 
\quad
U \in L^2_\Phi(\mathbb{C}), 
\label{equation:projector}
\end{equation}
where $C_\Phi=\lvert{B}\rvert^2/(2\pi \operatorname{Im}C)$, 
and $\Psi(z,\zeta)$ is a holomorphic quadratic function defined by the critical value of 
$-\{\phi(z,X)-\overline{\phi(\bar{\zeta},\bar{X})}\}/2i$ 
for $X\in\mathbb{C}$, that is, 
\begin{equation}
\Psi(z,\zeta)
=
\frac{\lvert{B}\rvert^2z\zeta}{4\operatorname{Im} C}
-
\frac{B^2z^2+\bar{B}^2\zeta^2}{8\operatorname{Im} C}
-
\frac{Az^2-\bar{A}\zeta^2}{4i}. 
\label{equation:Psi}
\end{equation}
In particular, $U=T_h{\circ}T_h^\ast U$ for $U \in \mathscr{H}_\Phi(\mathbb{C})$, 
and $\mathscr{H}_\Phi(\mathbb{C})$ becomes a reproducing kernel Hilbert space. 
\par
Here we recall elementary facts related with the classical Bargmann transform $B_h$  
which is the most important example of $T_h$. 
This was introduced by Bargmann in \cite{bargmann}. 
We can also refer \cite{folland} for this.  
The Bargmann transform $B_h$ on $\mathbb{R}$ is defined by 
$$
B_h(z)
=
2^{-1/2}(\pi h)^{-3/4}
\int_\mathbb{R} 
e^{-(z^2/4-zx+x^2/2)/h}
u(x)
dx, 
\quad
z\in\mathbb{C}. 
$$
Note that the integral kernel of $B_h$ is the generating function of Hermite functions. 
We denote $L_\Phi(\mathbb{C})$ and $\mathscr{H}_\Phi(\mathbb{C})$ for $B_h$ 
by $L^2_B(\mathbb{C})$ and $\mathscr{H}_B(\mathbb{C})$ respectively, those are, 
$$
L^2_B(\mathbb{C})
=
\left\{
U(z) \ \bigg\vert \ 
\int_\mathbb{C}
\lvert{U(z)}\rvert^2
e^{-\lvert{z}\rvert^2/2h}
L(dz)
< \infty
\right\},
$$
and 
$\mathscr{H}_B(\mathbb{C})=\{U(z) \in  L^2_B(\mathbb{C})\ \vert \ \partial U/\partial \bar{z}=0\}$. 
The Bargmann projector, which is the orthogonal projection of 
$L_B^2(\mathbb{C})$ onto $\mathscr{H}_B(\mathbb{C})$, is given by 
$$
B_h{\circ}B_h^\ast U(z)
=
\frac{1}{2\pi h}
\int_\mathbb{C}
e^{z\bar{\zeta}/2h}
U(\zeta)
e^{-\lvert\zeta\rvert^2/2h}
L(d\zeta), 
\quad
U \in L^2_B(\mathbb{C}).
$$ 
In view of the Taylor expansion of the reproducing kernel 
$e^{z\bar{\zeta}/2h}/(2\pi h)$, 
the formula $U=B_h{\circ}B_h^\ast U$ for $U \in \mathscr{H}_B(\mathbb{C})$ becomes 
$$
U(z)
=
\sum_{n=0}^\infty
(U,\varphi_{B,n})_{\mathscr{H}_B} \varphi_{B,n}(z), 
\quad
\varphi_{B,n}(z)
=
\frac{z^n}{\sqrt{\pi (2h)^{n+1} n!}},
\quad
n=0,1,2,\dotsc.
$$
A family of functions $\{\varphi_{B,n}\}_{n=0}^\infty$ 
is a complete orthonormal system of $\mathscr{H}_B(\mathbb{C})$ 
since $\mathscr{H}_B(\mathbb{C})$ is the set of all entire functions 
belonging to $L^2_B(\mathbb{C})$. 
\par
We shall see more detail about $\{\varphi_{B,n}\}_{n=0}^\infty$. 
We set for $U \in \mathscr{H}_B(\mathbb{C})$ 
$$
\Lambda_BU(z)
:=
2h
\frac{\partial U}{\partial z}(z), 
\quad
\Lambda_B^\ast U(z)
:=
zU(z). 
$$
Actually, $\Lambda_B^\ast$ is the adjoint of $\Lambda_B$ on $\mathscr{H}_B(\mathbb{C})$. 
Elementary computation gives 
\begin{equation}
(\Lambda_B^\ast\circ\Lambda_B+h)\varphi_{B,n}
=
(\Lambda_B\circ\Lambda_B^\ast-h)\varphi_{B,n}
=
(2n+1)h\varphi_{B,n},
\quad
n=0,1,2,\dotsc.
\label{equation:harmonic0} 
\end{equation}
We shall pull back these facts on $\mathbb{R}$ by using $B_h^\ast$. 
Set 
$$
\phi_{B,n}(x):=B_h^\ast\varphi_{B,n}(x),
\quad
P_B:=B_h^\ast \circ \Lambda_B \circ B_h, 
\quad
P_B^\ast:=B_h^\ast \circ \Lambda_B^\ast \circ B_h.   
$$
$\phi_{B,n}$ is said to be the $n$-th Hermite function, 
and a family $\{\phi_{B,n}\}_{n=0}^\infty$ is 
a complete orthonormal system of $L^2(\mathbb{R})$ 
since $B_h$ is a Hilbert space isomorphism of 
$L^2(\mathbb{R})$ onto $\mathscr{H}_B(\mathbb{C})$. 
Operators 
$$
P_B
=
h\frac{d}{dx}+x,
\quad
P_B^\ast
=
-h\frac{d}{dx}+x
$$
are said to be annihilation and creation operators respectively. 
Note that 
$$
\varphi_{B,n}=((2h)^nn!)^{-1/2}(\Lambda_B^\ast)^n\varphi_{B,0}, 
\quad
n=0,1,2,\dotsc,
$$
$$
\phi_{B,0}(x):=B_h^\ast\varphi_{B,0}(x)=\frac{1}{(\pi h)^{1/4}}e^{-x^2/2h}.
$$
Then we have so-call the Rodrigues formula
$$
\phi_{B,n}(x)
=
\frac{1}{\sqrt{(2h)^nn!}}
(P_B^\ast)^n
\phi_{B,0}(x)
=
\frac{(-1)^n}{(\pi h)^{1/4}\sqrt{(2h)^nn!}}
e^{-x^2/2h}
\left(h\frac{d}{dx}\right)^n
e^{-x^2/h}.
$$
Set $H_B=P_B^\ast \circ P_B + h = P_B \circ P_B^\ast - h$. Then 
$$
H_B= -h^2\frac{d^2}{dx^2}+x^2
$$
which is said to be the Hamiltonian of the harmonic oscillator. 
The equation \eqref{equation:harmonic0} becomes 
$$
H_B\phi_{B,n}=(2n+1)h\phi_{B,n}, 
\quad
n=0,1,2,\dotsc.
$$
Thus the $n$-th Hermite function $\phi_{B,n}$ is 
an eigenfunction of $H_B$ for the $n$-th eigenvalue $(2n+1)h$. 
\par
The purpose of the present paper is to study the generalization of the known facts 
on the usual Bargmann transform $B_h$. 
The plan of this paper is as follows. 
In Section~\ref{section:modified} 
we study the general Bargmann-type transform \eqref{equation:bargmann}, 
and obtain generalized annihilation and creation operators, 
the Hamiltonian of the generalized harmonic oscillator and its eigenvalues, 
generalized Hermite functions and the Rodrigues formula. 
In Section~\ref{section:noncommutative} 
we study a $2\times2$-system of second-order ordinary differential operators, 
which is said to be a non-commutative harmonic oscillators. 
More precisely, 
we study the commutative case of the non-commutative harmonic oscillators, 
and obtain the eigenvalues and eigenfunctions 
by using our original elementary computation. 
Finally in Section~\ref{section:ellipeses} 
we study the general Bargmann-type transform \eqref{equation:bargmann} 
which might be related with ellipses in the phase plane $\mathbb{R}^2$. 
%%%%%%
%%%%%% section 2
%%%%%%
\section{Modified harmonic oscillators and Hermite functions}
\label{section:modified}
In this section we study the general form of the Bargmann-type transform 
\eqref{equation:bargmann}. 
We remark that the choice of the constant $A$ in the phase function is not essential. 
We can choose 
$$
A
=
-
\frac{iB^2}{2\operatorname{Im} C}. 
$$ 
Then 
\eqref{equation:Phi} and \eqref{equation:Psi} become very simple as 
$$
\Phi(z)
=
\frac{\lvert{Bz}\rvert^2}{4\operatorname{Im} C},
\quad
\Psi(z,\bar{\zeta})
=\frac{Bz \cdot \overline{B\zeta}}{4\operatorname{Im} C} 
$$
respectively. 
Moreover, the orthogonal projector \eqref{equation:projector} 
of $L^2_\Phi(\mathbb{C})$ onto $\mathscr{H}_\Phi(\mathbb{C})$,  
and the I-Lagrangian submanifold \eqref{equation:kappa} become 
$$
T_h \circ T_h^\ast U(z)
=
\frac{C_\Phi}{h}
\int_{\mathbb{C}}
e^{Bz \cdot \overline{B\zeta}/2h \operatorname{Im} C}
U(\zeta)
e^{-\lvert{Bz}\rvert^2/2h \operatorname{Im} C}
L(d\zeta), 
$$
$$
\kappa_T(\mathbb{R}^2)
=
\left\{
\left(z,\frac{\lvert{B}\rvert^2}{2i\operatorname{Im} C}\bar{z}\right)
\ \bigg\vert \ 
z\in\mathbb{C}
\right\} 
$$
respectively. 
Recall $U(z)=T_h \circ T_h^\ast U(z)$ for all $U \in \mathscr{H}_\Phi(\mathbb{C})$. 
If we consider the Taylor expansion of 
$e^{-\lvert{Bz}\rvert^2/2h \operatorname{Im} C}$, 
we have for $U \in \mathscr{H}_\Phi(\mathbb{C})$, 
$$
U(z)
=
\sum_{n=0}^\infty
(U,\varphi_n)_{\mathscr{H}_\Phi} 
\varphi(z),
$$
$$
\varphi_n(z)
=
\frac{\lvert{B}\rvert}{\sqrt{2\pi h \operatorname{Im} C}}
\cdot
\frac{1}{\sqrt{n!}}
\cdot
\left(\frac{Bz}{\sqrt{2h \operatorname{Im} C}}\right)^n,
\quad
n=0,1,2,\dotsc.
$$
\begin{theorem}
\label{theorem:cons21} 
The family of monomials $\{\varphi_n\}_{n=0}^\infty$ 
is a complete orthonormal system of $\mathscr{H}_\Phi(\mathbb{C})$.  
\end{theorem}
\begin{proof}
The completeness is obvious. We have only to show that 
$(\varphi_m,\varphi_n)_{\mathscr{H}_\Phi}=\delta_{mn}$, 
where $\delta_{mn}$ is Kronecker's delta. 
Without loss of generality we may assume that $m \geqslant n$. 
By using the integration by parts and 
the change of variable 
$\zeta=Bz/\sqrt{h \operatorname{Im} C}$, we deduce that 
\begin{align*}
  (\varphi_m,\varphi_n)_{\mathscr{H}_\Phi}
& =
  \frac{\lvert{B}\rvert^2}{2\pi h \operatorname{Im} C}
  \cdot
  \frac{1}{\{m!n!(2h \operatorname{Im} C)^{m+n}\}^{1/2}}
\\
& \times
  \int_\mathbb{C}
  (Bz)^m(\overline{Bz})^n
  e^{-\lvert{Bz}\rvert^2/2h \operatorname{Im} C}
  L(dz)
\\
& =
  \frac{\lvert{B}\rvert^2}{2\pi h \operatorname{Im} C}
  \cdot
  \frac{1}{\{m!n!(2h \operatorname{Im} C)^{m+n}\}^{1/2}}
\\
& \times
  \int_\mathbb{C}
  \left\{
  \left(-\frac{2h \operatorname{Im} C}{\bar{B}} \frac{\partial}{\partial \bar{z}}\right)^m
  e^{-\lvert{Bz}\rvert^2/2h \operatorname{Im} C}
  \right\}
  (\overline{Bz})^n
  L(dz)
\\
& =
  \frac{\lvert{B}\rvert^2}{2\pi h \operatorname{Im} C}
  \cdot
  \frac{1}{\{m!n!(2h \operatorname{Im} C)^{m+n}\}^{1/2}}
\\
& \times
  \int_\mathbb{C}
  \left\{
  \left(-\frac{2h \operatorname{Im} C}{\bar{B}} \frac{\partial}{\partial \bar{z}}\right)^m
  (\overline{Bz})^n
  \right\}
  e^{-\lvert{Bz}\rvert^2/2h \operatorname{Im} C}
  L(dz)
\\
& =
  \delta_{mn}
  \frac{\lvert{B}\rvert^2}{2\pi h \operatorname{Im} C}
  \cdot
  \frac{1}{m!(2h \operatorname{Im} C)^m}
\\
& \times
  \int_\mathbb{C}
  \left\{
  \left(-\frac{2h \operatorname{Im} C}{\bar{B}} \frac{\partial}{\partial \bar{z}}\right)^m
  (\overline{Bz})^m
  \right\}
  e^{-\lvert{Bz}\rvert^2/2h \operatorname{Im} C}
  L(dz)
\\
& =
  \delta_{mn}
  \frac{\lvert{B}\rvert^2}{2\pi h \operatorname{Im} C}
  \int_\mathbb{C}
  e^{-\lvert{Bz}\rvert^2/2h \operatorname{Im} C}
  L(dz)
\\
& =
  \delta_{mn}
  \frac{1}{2\pi}
  \int_\mathbb{C}
  e^{-\lvert\zeta\rvert^2/2}
  L(d\zeta)
  =
  \delta_{mn}.
\end{align*}
This completes the proof.  
\end{proof}
Set $\phi_n(x)=T_h^\ast\varphi_n(x)$, $n=0,1,2,\dotsc$. 
Since $T_h$ is a Hilbert space isomorphism of $L^2(\mathbb{R})$ 
onto $\mathscr{H}_\Phi(\mathbb{C})$, we have the following. 
\begin{theorem}
\label{theorem:cons22}
$\{\phi_n\}_{n=0}^\infty$ is a complete orthonormal system of $L^2(\mathbb{R})$.  
\end{theorem}
In what follows we study the family of functions $\{\phi_n\}_{n=0}^\infty$ in detail. 
Let $\Lambda$ be a linear operator on $\mathscr{H}_\Phi(\mathbb{C})$ defined by 
$$
\Lambda U(z)
=
\frac{2h \operatorname{Im}C}{\lvert{B}\rvert^2}
\frac{dU}{dz}(z),
\quad
U \in \mathscr{H}_\Phi(\mathbb{C}). 
$$
Its Hilbert adjoint is 
$$
\Lambda^\ast U(z)
=
zU(z),
\quad
U \in \mathscr{H}_\Phi(\mathbb{C}). 
$$
We call $\Lambda$ and $\Lambda^\ast$ annihilation and creation operators on 
$\mathscr{H}_\Phi(\mathbb{C})$ respectively. 
Since $\varphi_n$ is a monomial of degree $n$, 
we have for $n=0,1,2,\dotsc$ 
\begin{equation}
\varphi_n(z)
=
\frac{1}{\sqrt{n!}}
\left(\frac{B}{\sqrt{2h\operatorname{Im}C}}\right)^n
\bigl(\Lambda^\ast\bigr)^n
\varphi_0(z),
\label{equation:rodrigues200} 
\end{equation}
\begin{equation}
\left(
\Lambda^\ast\circ\Lambda
+
\frac{h\operatorname{Im}C}{\lvert{B}\rvert^2}
\right)
\varphi_n(z)
=
\left(
\Lambda\circ\Lambda^\ast
-
\frac{h\operatorname{Im}C}{\lvert{B}\rvert^2}
\right)
\varphi_n(z)
=
\frac{h\operatorname{Im}C}{\lvert{B}\rvert^2}
(2n+1)
\varphi_n(z).
\label{equation:harmonic200} 
\end{equation}
We shall pull back these facts by using $T_h^\ast$. 
Set 
$$
P := T_h^\ast \circ \Lambda \circ T_h, 
\quad
P^\ast := T_h^\ast \circ \Lambda^\ast \circ T_h, 
$$
$$
H
:= 
P^\ast \circ P
+
\frac{h\operatorname{Im}C}{\lvert{B}\rvert^2}
=
P \circ P^\ast
-
\frac{h\operatorname{Im}C}{\lvert{B}\rvert^2}.
$$
\par
To state the concrete form of $H$, we introduce the Weyl pseudodifferential operators. 
For an appropriate function $a(x,\xi)$ of $(x,\xi)\in\mathbb{R}^2$, 
its Weyl quantization is defined by 
$$
\operatorname{Op}_h^\text{W}(a)u(x)
=
\frac{1}{2\pi h}
\iint_{\mathbb{R}^2}
e^{i(x-y)\xi/h}
a\left(\frac{x+y}{2},\xi\right)
u(y)
dyd\xi
$$
for $u \in \mathscr{S}(\mathbb{R})$, 
where $\mathscr{S}(\mathbb{R})$ denotes the Schwartz class on $\mathbb{R}$. 
Set $D_x=-id/dx$ for short. 
\par
Here we give the concrete forms of operators $P$, $P^\ast$ and $H$ on $\mathbb{R}$. 
\begin{proposition}
\label{theorem:pph200}
We have 
$$
P
=
-
\frac{1}{\bar{B}}
(hD_x+\bar{C}x), 
\quad
P^\ast
=
-
\frac{1}{B}
(hD_x+Cx), 
$$
\begin{align*}
  H
& =
  \frac{1}{\lvert{B}\rvert^2}
  \left\{
  h^2D_x^2
  +
  \lvert{C}\rvert^2x^2
  +
  (C+\bar{C})
  \left(
  xhD_x+\frac{h}{2i}
  \right)
  \right\}
\\
& =
  \frac{1}{\lvert{B}\rvert^2}
  \operatorname{Op}_h^\text{W}
  \Bigl(
  \xi^2
  +
  \lvert{C}\rvert^2x^2
  +
  (C+\bar{C})x\xi
  \Bigr).
\end{align*}
\end{proposition}
\begin{proof}
We first compute $P$ and $P^\ast$. 
Since $\Lambda \circ T_h = T_h \circ P$, we deduce that 
for any $u \in \mathscr{S}(\mathbb{R})$, 
\begin{align*}
   \Lambda \circ T_h u(z)
& =
  \frac{2h\operatorname{Im}C}{\lvert{B}\rvert^2}
  \frac{d}{dz}
  C_\phi 
  h^{-3/4}
  \int_\mathbb{R}
  e^{i\phi(z,x)/h}
  u(x)
  dx
\\
& =
  C_\phi 
  h^{-3/4}
  \int_\mathbb{R}
  e^{i\phi(z,x)/h}
  \left\{
  \frac{2h\operatorname{Im}C}{\lvert{B}\rvert^2}
  \frac{i}{h}
  \frac{\partial \phi}{\partial z}(z,x)
  \right\}
  u(x)
  dx
\\
& =
  C_\phi 
  h^{-3/4}
  \int_\mathbb{R}
  e^{i\phi(z,x)/h}
  \frac{2i\operatorname{Im}C}{\lvert{B}\rvert^2}
  \Bigl(
  Az+Bx
  \Bigr)
  u(x)
  dx
\\
& =
  C_\phi 
  h^{-3/4}
  \int_\mathbb{R}
  e^{i\phi(z,x)/h}
  \frac{2i\operatorname{Im}C}{\lvert{B}\rvert^2}
  \left(
  -\frac{iB^2}{2 \operatorname{Im}C}z+Bx
  \right)
  u(x)
  dx
\\
& =
  C_\phi 
  h^{-3/4}
  \int_\mathbb{R}
  e^{i\phi(z,x)/h}
  \frac{1}{\bar{B}}
  \Bigl(
  Bz+2i\operatorname{Im}Cx
  \Bigr)
  u(x)
  dx
\\
& =
  C_\phi 
  h^{-3/4}
  \int_\mathbb{R}
  \Bigl(
  (hD_x-Cx)
  e^{i\phi(z,x)/h}
  \Bigr)
  \frac{1}{\bar{B}}
  u(x)
  dx
\\
& +
  C_\phi 
  h^{-3/4}
  \int_\mathbb{R}
  e^{i\phi(z,x)/h}
  \frac{2i\operatorname{Im}Cx}{\bar{B}}
  u(x)
  dx
\\
& =
  C_\phi 
  h^{-3/4}
  \int_\mathbb{R}
  e^{i\phi(z,x)/h}
  \frac{-1}{\bar{B}}
  (hD_x+Cx)
  u(x)
  dx
\\
& +
  C_\phi 
  h^{-3/4}
  \int_\mathbb{R}
  e^{i\phi(z,x)/h}
  \frac{2i\operatorname{Im}Cx}{\bar{B}}
  u(x)
  dx
\\
& =
  C_\phi 
  h^{-3/4}
  \int_\mathbb{R}
  e^{i\phi(z,x)/h}
  \frac{-1}{\bar{B}}
  (hD_x+Cx-2i\operatorname{Im}Cx)
  u(x)
  dx
\\
& =
  C_\phi 
  h^{-3/4}
  \int_\mathbb{R}
  e^{i\phi(z,x)/h}
  \frac{-1}{\bar{B}}
  (hD_x+\bar{C}x)
  u(x)
  dx, 
\end{align*}
which shows that 
$P=-(hD_x+\bar{C}x)/\bar{B}$. 
In the same way, we can obtain 
$P^\ast=-(hD_x+Cx)/B$, 
which is certainly the adjoint of $P$ on $L^2(\mathbb{R})$. 
Next we compute $H$. Simple computation gives 
\begin{align*}
  H
& =
  P^\ast \circ P
  +
  \frac{h\operatorname{Im}C}{\lvert{B}\rvert^2}
  =
  \frac{1}{\lvert{B}\rvert^2}
  (hD_x+Cx)(hD_x+\bar{C}x)
  +
  \frac{h\operatorname{Im}C}{\lvert{B}\rvert^2}
\\
& =
  \frac{1}{\lvert{B}\rvert^2}
  \left\{
  h^2D_x^2
  +
  \lvert{C}\rvert^2x^2
  +
  (C+\bar{C})xhD_x
  +
  \frac{h\bar{C}}{i}
  +
  h\operatorname{Im}C
  \right\}
\\
& =
  \frac{1}{\lvert{B}\rvert^2}
  \left\{
  h^2D_x^2
  +
  \lvert{C}\rvert^2x^2
  +
  (C+\bar{C})xhD_x
  +
  \frac{h\operatorname{Re}C}{i}
  \right\}
\\
& =
  \frac{1}{\lvert{B}\rvert^2}
  \left\{
  h^2D_x^2
  +
  \lvert{C}\rvert^2x^2
  +
  (C+\bar{C})
  \left(
  xhD_x
  +
  \frac{h}{2i}
  \right)
  \right\}, 
\end{align*}
which completes the proof.
\end{proof}
By using the pull-back of 
\eqref{equation:rodrigues200} 
and 
\eqref{equation:harmonic200}, 
we have for $n=0,1,2,\dotsc$ 
\begin{align*}
  \phi_n(x)
& =
  \frac{1}{\sqrt{n!}}
  \left(\frac{B}{\sqrt{2h\operatorname{Im}C}}\right)^n
  \bigl(P^\ast\bigr)^n
  \phi_0(x)
\\
& =
  \frac{1}{\sqrt{n!}}
  \left(\frac{-1}{\sqrt{2h\operatorname{Im}C}}\right)^n
  \bigl(hD_x+Cx\bigr)^n
  \phi_0(x)
\\
& =
  \frac{1}{\sqrt{n!}}
  \left(\frac{-1}{\sqrt{2h\operatorname{Im}C}}\right)^n
  e^{-iCx^2/2h}
  (hD_x)^n
  \bigl(e^{iCx^2/2h}\phi_0(x)\bigr),
\end{align*}
$$
H\phi_n
=
\frac{h\operatorname{Im}{C}}{\lvert{B}\rvert^2}
(2n+1)\phi_n.
$$
If we compute the concrete form of $\phi_0$, 
then we obtain the Rodrigues formula for $\{\phi_n\}_{n=0}^\infty$. 
\begin{theorem}
\label{theorem:rodrigues202}
We have for $n=0,1,2,\dotsc$ 
\begin{align*}
  \phi_0(x)
& =
  \left(\frac{\operatorname{Im}C}{\pi h}\right)^{1/4}
  e^{-i\bar{C}x^2/2h},
\\
  \phi_n(x)
& =
  \left(\frac{\operatorname{Im}C}{\pi h}\right)^{1/4}
  \frac{1}{\sqrt{n!}}
  \left(\frac{-1}{\sqrt{2h\operatorname{Im}C}}\right)^n
  e^{-i\operatorname{Re}C x^2/2h}
  e^{\operatorname{Im}C x^2/2h}
  (hD_x)^n
  e^{-\operatorname{Im}C x^2/h}. 
\end{align*}
\end{theorem}
\begin{proof}
Recall the definition of $\phi_0$. 
We have 
\begin{align*}
  \phi_0(x)
& =
  T_h^\ast\varphi_0(x)
  =
  C_\phi h^{-3/4}
  \int_\mathbb{C}
  e^{-i\overline{\phi(z,x)}/h}
  \varphi_0(z)
  e^{-2\Phi(z)/h}
  L(dz)
\\
& =
  2^{-1}(\pi h)^{-5/4}(\operatorname{Im}C)^{-3/4}\lvert{B}\rvert^2
  \int_\mathbb{C}
  e^{F_1(z,x)/h}
  L(dz), 
\end{align*}
where 
$$
F_1(z,x)
=
-i
\left\{
\frac{i(\overline{Bz})^2}{4\operatorname{Im}C}
+
\overline{Bz}
+
\frac{\bar{C}x^2}{2}
\right\}
-
\frac{\lvert{Bz}\rvert^2}{2\operatorname{Im}C}.
$$
Change the variable $\zeta=\xi+i\eta:=\overline{Bz}$, $(\xi,\eta)\in\mathbb{R}^2$. 
We deduce 
$$
\phi_0(x)
=
2^{-1}(\pi h)^{-5/4}(\operatorname{Im}C)^{-3/4}
\int_\mathbb{C}
e^{F_2(\zeta,x)/h}
L(d\zeta), 
$$
\begin{align*}
  F_2(\zeta,x)
& =
  \frac{\zeta^2}{4\operatorname{Im}C}
  -
  i\zeta x
  -
  \frac{i\bar{C}x^2}{2}
  -
  \frac{\lvert{\zeta}\rvert^2}{2\operatorname{Im}C}
\\
& =
  -
  \frac{\bigl\{\xi-i(\eta-2\operatorname{Im}Cx)\bigr\}^2}{4\operatorname{Im}C}
  -
  \frac{(\eta-\operatorname{Im}Cx)^2}{\operatorname{Im}C}
  -
  \frac{i\bar{C}x^2}{2}. 
\end{align*}
Then we can obtain 
\begin{align*}
  \phi_0(x)
& =
  2^{-1}(\pi h)^{-5/4}(\operatorname{Im}C)^{-3/4}
  e^{-i\bar{C}x^2/2h}
  \int_\mathbb{R}
  e^{-\xi^2/4h\operatorname{Im}C}
  d\xi
  \int_\mathbb{R}
  e^{-\eta^2/h\operatorname{Im}C}
  d\eta
\\
& =
  \left(
  \frac{\operatorname{Im}C}{\pi h}
  \right)^{1/4}
  e^{-i\bar{C}x^2/2h}.
\end{align*}
This completes the proof.
\end{proof}
%%%%%%
%%%%%% section 3
%%%%%%
\section{The commutative case of non-commutative harmonic oscillators}
\label{section:noncommutative}
Consider a $2\times2$ system of second-order differential operators of the form 
$$
Q_{(\alpha,\beta)}
=
\frac{1}{2}
A_{(\alpha,\beta)}
\operatorname{Op}_h^\text{W}
(\xi^2+x^2)
+
J
\operatorname{Op}_h^\text{W}
(ix\xi),
$$ 
where 
$\alpha$ and $\beta$ are positive constants satisfying $\alpha\beta>1$, 
and 
$$
A_{(\alpha,\beta)}
=
\begin{bmatrix}
\alpha & 0
\\
0 & \beta 
\end{bmatrix}, 
\quad
J
=
\begin{bmatrix}
0 & -1
\\
1 & 0 
\end{bmatrix}. 
$$
A matrix $A_{(\alpha,\beta)}(\xi^2+x^2)/2+J(ix\xi)$, 
which is the symbol of the operator $Q_{(\alpha,\beta)}$, 
is a Hermitian matrix, and all its eigenvalues are real-valued. 
Note that all its eigenvalues are positive for $(x,\xi)\ne(0,0)$ 
if and only if $\alpha\beta>1$. 
In other words, $Q_{(\alpha,\beta)}$ is a system of 
semiclassical elliptic differential operators 
if and only if $\alpha\beta>1$. 
The system of differential operators $Q_{(\alpha,\beta)}$ 
was mathematically introduced in \cite{pw1} by Parmeggiani and Wakayama. 
They call $Q_{(\alpha,\beta)}$ a Hamiltonian of non-commutative harmonic oscillator. 
The word ``non-commutative'' comes from the non-commutativity 
$A_{(\alpha,\beta)}J \ne JA_{(\alpha,\beta)}$ for $\alpha\ne\beta$. 
It is not known that the system of differential equations for $Q_{(\alpha,\beta)}$ 
describes a physical phenomenon. 
\par
Parmeggiani and Wakayama intensively studied 
spectral properties of $Q_{(\alpha,\beta)}$ in 
\cite{pw1}, \cite{pw2} and \cite{pw3}. 
See also a monograph \cite{parmeggiani}. 
They proved that if $\alpha\beta>1$, 
then $Q_{(\alpha,\beta)}$ is a self-adjoint and positive operator, 
and its spectra consists of positive eigenvalues whose multiplicities are at most three. 
In case of $\alpha=\beta$, they obtained more detail. 
\par
The purpose of the present section is to give alternative proof 
of the results of Parmeggiani and Wakayama for the commutative case $\alpha=\beta$. 
More precisely, we study $Q_{(\alpha,\alpha)}$ 
by using the results in the previous section. 
\par
In what follows we assume that $\alpha=\beta$. 
Then $\alpha>1$ since $\alpha>0$ and $\alpha^2=\alpha\beta>1$. 
Let $I$ be the $2\times2$ identity matrix. 
Set $Q_\alpha=Q_{(\alpha,\alpha)}$ for short. 
Let $U$ be a $2\times2$ unitary matrix defined by 
$$
U
=
\frac{1}{\sqrt{2}}
\begin{bmatrix}
1 & -i
\\
1 & i 
\end{bmatrix}
$$ 
which diagonalize $J$ as 
$$
UJU^\ast
=
\begin{bmatrix}
-i & 0
\\
0 & i 
\end{bmatrix}.
$$
Then, we have
$$
U Q_\alpha U^\ast
=
\frac{\alpha}{2}
I
\operatorname{Op}_h^\text{W}
(\xi^2+x^2)
+
iUJU^\ast
\operatorname{Op}_h^\text{W}(x\xi) 
=
\begin{bmatrix}
H_{\alpha,+} & 0
\\
0 & H_{\alpha,-}
\end{bmatrix},
$$
$$
H_{\alpha,\pm}
=
\operatorname{Op}_h^\text{W}
(\xi^2+x^2)
\pm
\operatorname{Op}_h^\text{W}(x\xi)
=
\operatorname{Op}_h^\text{W}
\left(
\left\lvert
\sqrt{\frac{\alpha}{2}}
(\nu_{\alpha,\pm}\xi+x)
\right\rvert^2
\right),
$$
$$
\nu_{\alpha,\pm}
=
\frac{\pm1+i\sqrt{\alpha^2-1}}{\alpha}. 
$$
Note that $\lvert\nu_{\alpha,\pm}\rvert=1$ and $\operatorname{Im}\nu_{\alpha,\pm}>0$. 
\par
Here we make use of the results in the previous section by setting 
$$
B
=
\sqrt{\frac{2}{\alpha}}
\nu_{\alpha,\pm},
\quad
C=\nu_{\alpha,\pm},
\quad
A
=
\frac{B^2}{2i\operatorname{Im}C}
=
\frac{\nu_{\alpha,\pm}^2}{i\sqrt{\alpha^2-1}}. 
$$
Note that the requirement $\operatorname{Im}C>0$ is satisfied. 
Set  
\begin{align*}
  \phi_{\alpha,\pm,n}(x)
& =
  e^{\mp ix^2/2\alpha h}
  h_{\alpha,n}(x),
\\
  h_{\alpha,n}(x)
& =
  \left(\frac{\sqrt{\alpha^2-1}}{\pi h \alpha}\right)^{1/4}
  \frac{1}{\sqrt{n!}} 
  \left(-\sqrt{\frac{\alpha}{2\sqrt{\alpha^2-1}h}}\right)^n
\\
& \times
  e^{\sqrt{\alpha^2-1} x^2/2\alpha h}
  (hD_x)^n
  e^{-\sqrt{\alpha^2-1} x^2/\alpha h}
\end{align*}
for $n=0,1,2,\dotsc$. Then we deduce that 
$\{\phi_{\alpha,\pm,n}\}_{n=0}^\infty$ 
is a complete orthonormal system of $L^2(\mathbb{R})$, 
and 
$$
H_{\alpha,\pm}
\phi_{\alpha,\pm,n}
=
\frac{\sqrt{\alpha^2-1}}{2}h(2n+1)\phi_{\alpha,\pm,n},
\quad
n=0,1,2,\dotsc. 
$$
In order to get the eigenfunctions of $Q_\alpha$, we set 
$$
\Phi_{\alpha,+,n}(x)
=
U^\ast
\begin{bmatrix}
\phi_{\alpha,+,n} \\ 0
\end{bmatrix}, 
\quad
\Phi_{\alpha,-,n}(x)
=
U^\ast
\begin{bmatrix}
0 \\ \phi_{\alpha,-,n}
\end{bmatrix}, 
\quad
n=0,1,2,\dotsc, 
$$
those are,  
$$
\Phi_{\alpha,\pm,n}(x)
=
h_{\alpha,n}(x)
\cdot
\frac{e^{\mp ix^2/2\alpha h}}{\sqrt{2}}
\begin{bmatrix}
1 
\\
\pm i 
\end{bmatrix},
\quad
n=0,1,2,\dotsc.
$$
We have proved the results of this section as follows. 
\begin{theorem}
\label{theorem:pw} 
A system of $\mathbb{C}^2$-valued functions 
$\bigl\{\Phi_{\alpha,\mu,n} \ \big\vert \ \mu=\pm, n=0,1,2,\dotsc\bigr\}$ 
is a complete orthonormal system of $L^2(\mathbb{R}) \oplus L^2(\mathbb{R})$, 
and satisfies 
$$
Q_\alpha\Phi_{\alpha,\pm,n}
=
\frac{\sqrt{\alpha^2-1}}{2} h(2n+1)
\Phi_{\alpha,\pm,n},
\quad
n=0,1,2,\dotsc.
$$
\end{theorem}
This is not a new result. 
This was first proved by Parmeggiani and Wakayama in \cite{pw1}. 
We believe that our method of proof is easier than that of \cite{pw1}.
%%%%%%
%%%%%% section 4
%%%%%%
\section{Orthogonal systems associated with ellipses in the phase plane}
\label{section:ellipeses}
Throughout of the present section, we assume that $h=1$ for the sake of simplicity. 
We begin with recalling the relationship between the standard Bargmann transform $B_1$ 
and circles in the phase plane. 
Here we introduce a Berezin-Toeplitz quantization on $\mathbb{C}$. 
Let $b(z)$ be an appropriate function on $\mathbb{C}$. Set 
$$
a(x,\xi)
=
\frac{1}{\pi}
\iint_{\mathbb{R}^2}
e^{-(x-y)^2-(\xi-\eta)^2}
b(y-i\eta)
dyd\eta. 
$$
It is known that 
$$
\left(\operatorname{Op}^\text{W}_1(a)u,v\right)_{L^2}
=
(bB_1u,B_1v)_{L^2_B}
=
\left(\tilde{T}_bB_1u,B_1v\right)_{\mathscr{H}_B}
$$
for $u,v \in \mathscr{S}(\mathbb{R})$. 
See, e.g., \cite{sjoestrand} and \cite{chihara1}. 
The operator $\tilde{T}_b$ is said to be the Berezin-Toeplitz quantization of $b$, 
which acts on $\mathscr{H}_B(\mathbb{C})$. 
If $b$ is a characteristic function on $\mathbb{C}$, 
then $\operatorname{Op}^\text{W}_1(a)$ 
is said to be a Daubechies' localization operator introduced in \cite{daubechies}. 
Moreover Daubechies proved that if $b(z)$ is radially symmetric, that is, 
$b$ is of the form $b(x-i\xi)=c(x^2+\xi^2)$ with some function $c(s)$ for $s\geqslant0$, 
then all the usual Hermite functions $\phi_{B,n}$ ($n=0,1,2,\dotsc$) 
are the eigenfunctions of $\tilde{T}_b$:
$$
\tilde{T}_b\phi_{B,n}=\lambda_n\phi_{B,n}, 
\quad
\lambda_n
=
\frac{1}{n!}
\int_0^\infty
c(2s)s^ne^{-s}
ds, 
\quad
n=0,1,2,\dotsc.
$$ 
\par
Recently, Daubechies' results have been developed. 
Here we quote two interesting results of inverse problems 
studied in \cite{ad} and \cite{yoshino}. 
On one hand, in \cite{yoshino} Yoshino proved that radially symmetric symbols 
of the Berezin-Toeplitz quantization on $\mathbb{C}$ can be reconstructed 
by all the eigenvalues $\{\lambda_n\}_{n=0}^\infty$. 
More precisely, by using the framework of hyperfunctions, 
he obtained the reconstruction formula for radially symmetric symbols. 
On the other hand, in \cite{ad} Abreu and D\"orfler studied 
the inverse problem for Daubechies' localization operators. 
Let $\Omega$ be a bounded subset of $\mathbb{C}$, 
and let $b(z)$ be the characteristic function of $\Omega$. 
They proved that if there exists a nonnegative integer $m$ such that 
the $m$-th Hermite function $\phi_{B,m}$ 
is an eigenfunction of $\operatorname{Op}^\text{W}_1(a)$, 
then $\Omega$ must be a disk centered at the origin. 
In this case it follows automatically that 
all the Hermite functions $\phi_{B,n}$ are eigenfunctions of 
$\operatorname{Op}^\text{W}_1(a)$ 
associated with eigenvalues 
$$
\lambda_n
=
e^{-R}
\sum_{k=n+1}^\infty
\frac{R^k}{k!},
\quad
n=0,1,2,\dotsc, 
$$
respectively, where $R$ is the radius of $\Omega$. 
In particular $R=-\log(1-\lambda_0)$. 
That is the review of the relationship between 
the usual Bargmann transform and circles (or disks) in $\mathbb{C}$. 
\par
The purpose of the present section is 
to consider the possibility of the extension of the above 
to ellipses (or elliptic disks) in $\mathbb{C}$. 
Unfortunately, however, we could not obtain the extension of the above. 
In what follows we introduce a family of functions 
which might be concerned with ellipses in $\mathbb{C}$. 
Here it is worth to mention the interesting work \cite{EM} of van Eijndhoven and Meyers. 
They introduced for $0<s<1$ function spaces $\chi_s(\mathbb{C})$, 
which is the set of all entire functions $\varphi(z)$ on $\mathbb{C}$ satisfying 
$$
\int_\mathbb{C}
\lvert\varphi(z)\rvert^2
\exp\left(-\frac{1-s^2}{2s}\lvert{z}\rvert^2+\frac{1+s^2}{4s}(z^2+\bar{z}^2)\right)
L(dz)
<
\infty.
$$
As the author pointed out in \cite{chihara2}, 
$\chi_s(\mathbb{C})$ is determined by the ellipse on $\mathbb{C}$ of the form 
$$
\{x-i\xi \in \mathbb{C}\ \vert \ (x,\xi)\in\mathbb{R}^2, sx^2+\xi^2=\rho^2\}, 
\quad
\rho>0,
$$
and $\chi_s(\mathbb{C})$ is a special case of $\mathscr{H}_\Phi(\mathbb{C})$ with 
$$
A
=
\frac{i}{s},
\quad
B
=
\pm 
i\sqrt{1-s^2},
\quad
C
=
t+is
\quad
(t\in\mathbb{R}). 
$$
Recently Ali, G\'orska, Horzela and Szafraniec in \cite{AGHS} 
studied some kinds of generating functions of Hermite polynomials 
in the abstract setting, 
and introduced some ortonormalized holomorphic Hermite functions 
in some function spaces including $\chi_s(\mathbb{C})$. 
Here we introduce a holomorphic Hermite functions 
$\{h_n\}_{n=0}^\infty$ on $\mathbb{C}$ and normalizing constants $b(n)$ 
defined by 
$$
h_n(z)
=
e^{z^2/2}
\left(-\frac{d}{dz}\right)^n
e^{-z^2},
\quad
b(n)
=
\frac{\pi\sqrt{s}}{1-s}
\left(2\frac{1+s}{1-s}\right)^n
n!. 
$$
One of the interesting results of \cite{AGHS} is that 
$\{h_n/\sqrt{b(n)}\}_{n=0}^\infty$ 
is a complete orthonormal system of $\chi_s(\mathbb{C})$. 
See \cite{chihara3} for more information on 
general holomorphic Hermite functions and their basic properties.  
\par
Let $\alpha>0$ and let $\beta\in\mathbb{R}$. 
Suppose that $(\alpha,\beta)\ne(1,0)$. 
For $\rho>0$, 
$$
E_\rho
:=
\{
x-i\xi\in\mathbb{C}
\ \vert \ 
(x,\xi)\in\mathbb{R}^2, 
\lvert\alpha{x}-i(\beta{x}+\xi)\rvert \leqslant \rho
\}
$$
is an elliptic disk in $\mathbb{C}$. 
Note that $E_\rho$ is a usual disk if and only if $(\alpha,\beta)=(1,0)$. 
Note that 
$$
\bigl\{
\partial{E_\rho}
\ \big\vert \ 
\alpha>0, \beta\in\mathbb{R}, \rho>0 
\bigr\} 
$$
is the set of all ellipses centered at the origin, 
where 
$\partial{E_\rho}=\{x-i\xi\in\mathbb{C}\ \vert \ \lvert\alpha{x}-i(\beta{x}+\xi)\rvert = \rho\}$. Indeed, consider a function
$$
F(x,y;a,b,\theta)
:=
a(x\cos\theta-\xi\sin\theta)^2
+
b(x\sin\theta+\xi\cos\theta)^2, 
\quad
a>0, b>0, \theta\in[0,2\pi]. 
$$
Elementary computation gives 
$$
\frac{F(x,y;a,b,\theta)}{a\sin^2\theta+b\cos^2\theta}
=
\frac{ab}{(a\sin^2\theta+b\cos^2\theta)^2}
x^2
+
\left\{
\frac{(b-a)\sin\theta\cos\theta}{a\sin^2\theta+b\cos^2\theta}x
+
\xi
\right\}^2.
$$
\par
Here we introduce a function $\psi_0(z)$ which seems 
to be related with an elliptic disk $E_\rho$. 
Set $z=x-i\xi$ and $\zeta=\alpha{x}-i(\beta{x}+\xi)$ 
for $(x,\xi)\in\mathbb{R}^2$, $\alpha>0$ and $\beta\in\mathbb{R}$. 
Then 
$$
\zeta
=
\frac{\alpha+1-i\beta}{2}
z
+
\frac{\alpha-1-i\beta}{2}
\bar{z},
\quad
z
=
\frac{\alpha+1+i\beta}{2\alpha}
\zeta
-
\frac{\alpha-1-i\beta}{2\alpha}
\bar{\zeta}.
$$
We define the function $\psi_0(z)$ by 
$$
\psi_0(z)
=
\exp\left(-\frac{a}{4}z^2\right),
\quad
a
=
\frac{\alpha^2+\beta^2-1+2i\beta}{\alpha^2+\beta^2+1}. 
$$
Let $\lVert\cdot\rVert_{\mathscr{H}_B}$ be the norm of $\mathscr{H}_B(\mathbb{C})$ 
determined by the inner product $(\cdot,\cdot)_{\mathscr{H}_B}$. 
The properties of $\psi_0(z)$ are the following. 
\begin{lemma}
\label{theorem:psizero}
We have 
\begin{itemize}
\item[{\rm (i)}] 
$\psi_0\in\mathscr{H}_B(\mathbb{C})$.  
\item[{\rm (ii)}] 
$\lvert\psi_0(z)\rvert^2e^{-\lvert{z}\rvert^2/2}=e^{-\lvert\zeta\rvert^2/(\alpha^2+\beta^2+1)}$. 
\item[{\rm (iii)}] 
$\lVert\psi_0\rVert_{\mathscr{H}_B}^2=(\alpha^2+\beta^2+1)\pi/\alpha$. 
\end{itemize}
\end{lemma}
\begin{proof} 
We first show (i). 
We have only to show the integrability of 
$\lvert\psi_0(z)\rvert^2e^{-\lvert{z}\rvert^2/2}$ 
since $\psi_0(z)$ is an entire function.  
Note that 
$$
\lvert{a}\rvert^2
=
\frac{(\alpha^2+\beta^2-1)^2+4\beta^2}{(\alpha^2+\beta^2+1)^2} 
=
1
-
\left(
\frac{2\alpha}{\alpha^2+\beta^2+1}
\right)^2.
$$
We have 
$$
0
<
\frac{2\alpha}{\alpha^2+\beta^2+1}
<1
$$
since 
$$
\alpha^2+\beta^2+1-2\alpha
=
(\alpha-1)^2+\beta^2
>0
$$
for $(\alpha,\beta)\ne(1,0)$. 
Thus $0<\lvert{a}\rvert<1$. 
We deduce that 
\begin{align*}
  \lvert\psi_0(z)\rvert^2
  e^{-\lvert{z}\rvert^2/2}
& \leqslant
  \exp
  \left(
  -\frac{a}{4}z^2-\frac{\bar{a}}{4}\bar{z}^2
  -\frac{1}{2}\lvert{z}\rvert^2
  \right)
  \leqslant
  \exp
  \left(
  -\frac{1-\lvert{a}\rvert}{2}\lvert{z}\rvert^2
  \right)
\\
& =
  \exp
  \left\{
  -\frac{1}{2}
  \left(
  1
  -
  \sqrt{1-\left(\frac{2\alpha}{\alpha^2+\beta^2+1}\right)^2}
  \right)
  \lvert{z}\rvert^2
  \right\}.  
\end{align*}
This implies that 
$\lvert\psi_0(z)\rvert^2e^{-\lvert{z}\rvert^2/2}$ 
is integrable on $\mathbb{C}$ with respect to the Lebesgue measure $L(dz)$ 
and $\psi_0\in\mathscr{H}_B(\mathbb{C})$. 
\par
We show (ii) and (iii). 
Elementary computation shows that 
$$
\frac{\lvert\zeta\rvert^2}{\alpha^2+\beta^2+1}
=
\frac{\lvert{z}\rvert^2}{2}
+
\frac{a}{4}z^2
+
\frac{\bar{a}}{4}\bar{z}^2, 
$$
which implies (ii). Moreover, it is easy to see that 
$dz \wedge d\bar{z}=\alpha^{-1}d\zeta \wedge d\bar{\zeta}$ 
and 
$$
\lVert\psi_0\rVert_{\mathscr{H}_B}^2
=
\int_\mathbb{C}
\lvert\psi_0(z)\rvert^2
e^{-\lvert{z}\rvert^2/2}
L(dz)
=
\frac{1}{\alpha}
\int_\mathbb{C}
e^{-\lvert\zeta\rvert^2/(\alpha^2+\beta^2+1)}
L(d\zeta)
=
\frac{\alpha^2+\beta^2+1}{\alpha}
\pi.
$$
This completes the proof.
\end{proof}
The identity 
$\lvert\psi_0(z)\rvert^2e^{-\lvert{z}\rvert^2/2}=e^{-\lvert\zeta\rvert^2/(\alpha^2+\beta^2+1)}$ 
makes us to expect that $\psi_0$ might be related with an elliptic disk $E_\rho$ 
and generate a family of eigenfunctions for the Daubechies' localization operators 
supported in $E_\rho$. 
Unfortunately, however, this expectation fails to hold. 
The purpose of the present section is to generate a family of functions by $\psi_0$, 
and show its properties similar to the previous sections. 
To state our results in the present section, 
we here introduce some notation. 
Set 
$$
\lambda
=
\frac{2\alpha^2}{(\alpha^2+\beta^2+1)(\alpha^2+\beta^2-1-2i\beta)}, 
\quad 
\Lambda_{\alpha,\beta}
=
\frac{1}{a}
\frac{\partial}{\partial z}
+
\frac{z}{2}, 
\quad
\Lambda_{\alpha,\beta}^\ast
=
\frac{\partial}{\partial z}
+
\frac{a+2\lambda}{2}z. 
$$
It is easy to see that $a+2\lambda=1/\bar{a}$, $\Lambda_{\alpha,\beta}\psi_0=0$,  
and $\Lambda_{\alpha,\beta}^\ast$ is the Hilbert adjoint of 
$\Lambda_{\alpha,\beta}$ on $\mathscr{H}_B(\mathbb{C})$. 
We make use of $\psi_0$, $\Lambda_{\alpha,\beta}$ and $\Lambda_{\alpha,\beta}^\ast$ 
as a generating element of a family of functions, 
and annihilation and creation operators respectively. 
Set $\psi_n=(\Lambda_{\alpha,\beta}^\ast)^n\psi_0$ for $n=0,1,2,\dotsc$, and set 
$$
C_{\alpha,\beta}
=
\frac{1-\bar{a}}{2\bar{a}}
=
\frac{1+i\beta}{\alpha^2+\beta^2-1-2i\beta} 
$$
for short. 
Properties of 
$\Lambda_{\alpha,\beta}$, 
$\Lambda_{\alpha,\beta}^\ast$ 
and 
$\{\psi_n\}_{n=0}^\infty$ are the following. 
\begin{theorem}
\label{theorem:psin}
\quad
\begin{itemize}
\item[{\rm (i)}] 
$\{\psi_n\}_{n=0}^\infty$ satisfies a formula of the form 
$$
\psi_n(z)
=
\left\{
e^{-\lambda z^2/2}
\left(\frac{\partial}{\partial z}\right)^n
e^{\lambda z^2/2}
\right\}
\psi_0(z),
\quad
n=0,1,2,\dotsc.
$$
\item[{\rm (ii)}] 
For $n=1,2,3,\dotsc$, 
$$
\Lambda_{\alpha,\beta}
(\Lambda_{\alpha,\beta}^\ast)^n
=
(\Lambda_{\alpha,\beta}^\ast)^n
\Lambda_{\alpha,\beta}
+
n
\frac{\lambda}{a}
(\Lambda_{\alpha,\beta}^\ast)^{n-1}.
$$
\item[{\rm (iii)}] 
$\{\psi_n\}_{n=0}^\infty$ is a complete orthogonal system of $\mathscr{H}_B(\mathbb{C})$.
\item[{\rm (iv)}] 
For $n=0,1,2,\dotsc$, 
$$
\left(
\frac{\Lambda_{\alpha,\beta}^\ast\Lambda_{\alpha,\beta}}{\lvert{C_{\alpha,\beta}}\rvert^2}
+
\frac{\alpha^2}{1+\beta^2}
\right)
\psi_n
=
\frac{\alpha^2}{1+\beta^2}
(2n+1)
\psi_n.
$$
\end{itemize}
\end{theorem}
\begin{proof}
First we show (i). 
Note that for any $c\in\mathbb{C}$ and for any holomorphic function $f(z)$, 
we deduce that 
\begin{align*}
  \left(
  \frac{\partial}{\partial z} 
  +
  cz
  \right)
  f(z)
& =
  \left( 
  \frac{\partial}{\partial z}
  +
  cz
  \right)
  \bigl\{
  e^{-cz^2/2}
  \bigl(e^{cz^2/2}f(z)\bigr)
  \bigr\}
\\
& =
  -cze^{-cz^2/2}
  \bigl(e^{cz^2/2}f(z)\bigr)
  +
  e^{-cz^2/2}
  \frac{\partial}{\partial z}
  \bigl(e^{cz^2/2}f(z)\bigr)
  +
  cze^{-cz^2/2}
  \bigl(e^{cz^2/2}f(z)\bigr)
\\
& =
  e^{-cz^2/2}
  \frac{\partial}{\partial z}
  \bigl(e^{cz^2/2}f(z)\bigr).  
\end{align*}
Using this repeatedly, we have 
\begin{align*}
  \psi_n(z)
& =
  (\Lambda_{\alpha,\beta}^\ast)^n
  \psi_0(z)
  =
  \left\{
  \frac{\partial}{\partial z}
  +
  \left(\frac{a}{2}+\lambda\right)
  \right\}^n
  e^{-az^2/4}
\\
& =
  e^{-az^2/4-\lambda z^2/2}
  \left(\frac{\partial}{\partial z}\right)^n
  \bigl(e^{az^2/4+\lambda z^2/2} \cdot e^{-az^2/4}\bigr)
  =
  e^{-az^2/4-\lambda z^2/2}
  \left(\frac{\partial}{\partial z}\right)^n
  e^{\lambda z^2/2}
\\
& =
  \left\{
  e^{-\lambda z^2/2}
  \left(\frac{\partial}{\partial z}\right)^n
  e^{\lambda z^2/2}
  \right\}
  e^{-az^2/4}
  =
  \left\{
  e^{-\lambda z^2/2}
  \left(\frac{\partial}{\partial z}\right)^n
  e^{\lambda z^2/2}
  \right\}
  \psi_0(z), 
\end{align*}
which is desired. 
\par
Next we show (ii). We employ induction on $n$. 
For $n=1$, we deduce that 
\begin{align*}
  \Lambda_{\alpha,\beta}\Lambda_{\alpha,\beta}^\ast
  -
  \Lambda_{\alpha,\beta}^\ast\Lambda_{\alpha,\beta}
& =
  \frac{1}{a}
  \left(
  \frac{\partial}{\partial z}
  +
  \frac{a}{2}z
  \right)
  \left\{
  \frac{\partial}{\partial z}
  +
  \left(
  \frac{a}{2}+\lambda
  \right)
  z
  \right\}
\\
& -
  \left\{
  \frac{\partial}{\partial z}
  +
  \left(
  \frac{a}{2}+\lambda
  \right)
  z
  \right\}
  \frac{1}{a}
  \left(
  \frac{\partial}{\partial z}
  +
  \frac{a}{2}z
  \right)
\\
& =
  \frac{1}{a}
  \left(
  \frac{\partial}{\partial z}
  +
  \frac{a}{2}z
  \right)
  (\lambda z)
  -
  (\lambda z)
  \frac{1}{a}
  \left(
  \frac{\partial}{\partial z}
  +
  \frac{a}{2}z
  \right)
  (\lambda z)
\\
& =
  \frac{\lambda}{a}
  =
  1
  \cdot
  \frac{\lambda}{a}
  \cdot
  (\Lambda_{\alpha,\beta}^\ast)^0.
\end{align*}
Here we suppose that (ii) holds for some $n=1,2,3,\dotsc$. 
We show the case of $n+1$. 
By using the cases of $n$ and $1$, we deduce that 
\begin{align*}
  \Lambda_{\alpha,\beta}(\Lambda_{\alpha,\beta}^\ast)^{n+1}
& =
  \Lambda_{\alpha,\beta}(\Lambda_{\alpha,\beta}^\ast)^n\Lambda_{\alpha,\beta}^\ast
\\
& =
  \left\{
  (\Lambda_{\alpha,\beta}^\ast)^n
  \Lambda_{\alpha,\beta}
  +
  n
  \frac{\lambda}{a}
  (\Lambda_{\alpha,\beta}^\ast)^{n-1}
  \right\}
  \Lambda_{\alpha,\beta}^\ast
\\
& =
  (\Lambda_{\alpha,\beta}^\ast)^n
  (\Lambda_{\alpha,\beta}\Lambda_{\alpha,\beta}^\ast)
  +
  n
  \frac{\lambda}{a}
  (\Lambda_{\alpha,\beta}^\ast)^n
\\
& =
  (\Lambda_{\alpha,\beta}^\ast)^n
  \left\{
  \Lambda_{\alpha,\beta}^\ast\Lambda_{\alpha,\beta}
  +
  \frac{\lambda}{a}
  \right\}
  +
  n
  \frac{\lambda}{a}
  (\Lambda_{\alpha,\beta}^\ast)^n
\\
& =
  (\Lambda_{\alpha,\beta}^\ast)^{n+1}
  \Lambda_{\alpha,\beta}
  +
  (n+1)
  \frac{\lambda}{a}
  (\Lambda_{\alpha,\beta}^\ast)^n,  
\end{align*}
which is desired. 
\par
For (iii), we here show only the orthogonality 
\begin{align}
& (\psi_m,\psi_n)_{\mathscr{H}_B}
  =
  \left(\frac{\lambda}{a}\right)^n
  n!
  \lVert\psi_0\rVert_{\mathscr{H}_B}^2
  \delta_{mn}
\nonumber
\\
  =
& \frac{(\alpha^2+\beta^2+1)\pi}{\alpha}
  \left(
  \frac{2\alpha^2}{(\alpha^2+\beta^2-1)^2+4\beta^2}
  \right)^n
  n!
  \delta_{mn}, 
  \quad
  m,n=0,1,2,\dotsc.
\label{equation:orthogonality} 
\end{align}
The completeness will be automatically proved later. 
Recall that $\Lambda_{\alpha,\beta}\psi_0=0$. 
Suppose that $m \geqslant n$. 
By using (ii) repeatedly, we deduce that 
\begin{align*}
  (\psi_m,\psi_n)_{\mathscr{H}_B}
& =
  \bigl(
  (\Lambda_{\alpha,\beta}^\ast)^m\psi_0,
  (\Lambda_{\alpha,\beta}^\ast)^n\psi_0
  \bigr)_{\mathscr{H}_B} 
\\
& =
  \bigl(
  (\Lambda_{\alpha,\beta}^\ast)^{m-1}\psi_0,
   \Lambda_{\alpha,\beta}(\Lambda_{\alpha,\beta}^\ast)^n\psi_0
  \bigr)_{\mathscr{H}_B} 
\\
& =
  \left(
  (\Lambda_{\alpha,\beta}^\ast)^{m-1}\psi_0,
  (\Lambda_{\alpha,\beta}^\ast)^n\Lambda_{\alpha,\beta}\psi_0
  +
  n\frac{\lambda}{a}(\Lambda_{\alpha,\beta}^\ast)^{n-1}\psi_0
  \right)_{\mathscr{H}_B}
\\
& =
  n\frac{\lambda}{a}
  \bigl(
  (\Lambda_{\alpha,\beta}^\ast)^{m-1}\psi_0,
  (\Lambda_{\alpha,\beta}^\ast)^{n-1}\psi_0
  \bigr)_{\mathscr{H}_B}
\\
& =
  \dotsb
\\
& =
  n!\left(\frac{\lambda}{a}\right)^n
  \bigl(
  (\Lambda_{\alpha,\beta}^\ast)^{m-n}\psi_0,
   \psi_0
  \bigr)_{\mathscr{H}_B}. 
\end{align*}
If $m>n$, then 
$$
(\psi_m,\psi_n)_{\mathscr{H}_B}
=
n!\left(\frac{\lambda}{a}\right)^n
\bigl(
(\Lambda_{\alpha,\beta}^\ast)^{m-n-1}\psi_0,
\Lambda_{\alpha,\beta}\psi_0
\bigr)_{\mathscr{H}_B}
=
0.
$$
If $m=n$, then 
$$
(\psi_n,\psi_n)_{\mathscr{H}_B}
=
n!\left(\frac{\lambda}{a}\right)^n
(\psi_0,\psi_0)_{\mathscr{H}_B}
=
n!\left(\frac{\lambda}{a}\right)^n
\lVert\psi_0\rVert_{\mathscr{H}_B}^2. 
$$
\par
Finally we show (iv). 
By using (ii) and $\Lambda_{\alpha,\beta}\psi_0=0$ again, 
we deduce that 
\begin{align*}
  \left(
  \frac{\Lambda_{\alpha,\beta}^\ast\Lambda_{\alpha,\beta}}{\lvert{C_{\alpha,\beta}}\rvert^2}
  +
  \frac{\alpha^2}{1+\beta^2}
  \right)
  \psi_n 
& = 
  \frac{\Lambda_{\alpha,\beta}^\ast\Lambda_{\alpha,\beta}(\Lambda_{\alpha,\beta}^\ast)^n}{\lvert{C_{\alpha,\beta}}\rvert^2}
  \psi_0
  +
  \frac{\alpha^2}{1+\beta^2}
  \psi_n 
\\
& = 
  \frac{1}{\lvert{C_{\alpha,\beta}}\rvert^2}
  \left\{
  (\Lambda_{\alpha,\beta}^\ast)^{n+1}\Lambda_{\alpha,\beta}
  +
  n\frac{\lambda}{a}
  (\Lambda_{\alpha,\beta}^\ast)^n
  \right\}
  \psi_0
  +
  \frac{\alpha^2}{1+\beta^2}
  \psi_n
\\
& =
  \frac{n\lambda}{a\lvert{C_{\alpha,\beta}}\rvert^2}
  (\Lambda_{\alpha,\beta}^\ast)^n
  \psi_0
  +
  \frac{\alpha^2}{1+\beta^2}
  \psi_n
\\
& =
  \left\{
  \frac{n\lambda}{a\lvert{C_{\alpha,\beta}}\rvert^2}
  +
  \frac{\alpha^2}{1+\beta^2}
  \right\}
  \psi_n
  =
  \frac{\alpha^2}{1+\beta^2}
  (2n+1)
  \psi_n.
\end{align*}
This completes the proof. 
\end{proof}
Here we introduce a family of functions $\{\Psi_n\}_{n=0}^\infty$ on $\mathbb{R}$ 
by setting $\Psi_n=B_1^\ast\psi_n$ for $n=0,1,2,\dotsc$. 
In order to study basic properties on $\{\Psi_n\}_{n=0}^\infty$, 
we introduce notation. 
Set 
$$
P_{\alpha,\beta}
:=
\frac{1}{\overline{C_{\alpha,\beta}}}
B_1^\ast \circ \Lambda_{\alpha,\beta} \circ B_1,
\quad
P_{\alpha,\beta}^\ast
:=
\frac{1}{C_{\alpha,\beta}}
B_1^\ast \circ \Lambda_{\alpha,\beta}^\ast \circ B_1, 
\quad
H_{\alpha,\beta}
:=
P_{\alpha,\beta}^\ast \circ P_{\alpha,\beta}
+
\frac{\alpha^2}{1+\beta^2} 
$$
for short. 
Then we have 
$$
P_{\alpha,\beta}
=
\frac{d}{dx}
+
\frac{\alpha^2+i(\alpha^2+\beta^2+1)\beta}{1+\beta^2}x,
\quad
P_{\alpha,\beta}^\ast
=
-
\frac{d}{dx}
+
\frac{\alpha^2-i(\alpha^2+\beta^2+1)\beta}{1+\beta^2}x,
$$
\begin{align*}
  H_{\alpha,\beta}
& =
  -
  \frac{d^2}{dx^2}
  +
  \frac{\alpha^4+(\alpha^2+\beta^2+1)\beta^2}{(1+\beta^2)^2}x^2
  -
  2i
  \frac{(\alpha^2+\beta^2+1)\beta}{1+\beta^2}
  x\frac{d}{dx}
  -
  i
  \frac{(\alpha^2+\beta^2+1)\beta}{1+\beta^2}
\\
& =
  \operatorname{Op}^\text{W}
  \left(
  \xi^2
  +
  \frac{\alpha^4+(\alpha^2+\beta^2+1)\beta^2}{(1+\beta^2)^2}x^2
  +
  2
  \frac{(\alpha^2+\beta^2+1)\beta}{1+\beta^2}
  x\xi
  \right).
\end{align*}
Set 
$$
A_{\alpha,\beta}
=
\pi^{1/4}
\left\{
\frac{\alpha^2+\beta^2+1}{1-i\beta}
\right\}^{1/2}, 
$$
where we take its argument in $(-\pi/4,\pi/4)$.  
Our results in the present section are the following. 
\begin{theorem}
\label{theorem:PSIN}
\quad
\begin{itemize}
\item[{\rm (i)}] 
We have for $n=0,1,2,\dotsc$ 
\begin{align*}
  \Psi_n(x)
& =
  A_{\alpha,\beta}
  (-C_{\alpha,\beta})^n
  \exp\left(
      -i
      \frac{(\alpha^2+\beta^2+1)\beta}{2(1+\beta^2)}
      x^2
      \right)
\\
& \times
  \exp\left(
      \frac{\alpha^2}{2(1+\beta^2)}
      x^2
      \right) 
  \left(\frac{d}{dx}\right)^n
  \exp\left(
      -
      \frac{\alpha^2}{1+\beta^2}
      x^2
      \right). 
\end{align*}
In particular, 
$$
\Psi_0(x)
=
A_{\alpha,\beta}
\exp\left(
    -
    \frac{\alpha^2+i(\alpha^2+\beta^2+1)\beta}{2(1+\beta^2)}
    x^2
    \right). 
$$
\item[{\rm (ii)}] 
$\{\Psi_n\}_{n=0}^\infty$ is a family of eigenfunctions of $H_{\alpha,\beta}$, 
that is, 
$$
H_{\alpha,\beta}\Psi_n
=
\frac{\alpha^2}{1+\beta^2}
(2n+1)
\Psi_n,
\quad
n=0,1,2,\dotsc.
$$
\item[{\rm (iii)}] 
$\{\Psi_n\}_{n=0}^\infty$ 
is a complete orthogonal system of $L^2(\mathbb{R})$.  
\end{itemize}
\end{theorem}
Recall that Theorem~\ref{theorem:psin} was proved 
except for the completeness of $\{\psi_n\}_{n=0}^\infty$.  
Theorem~\ref{theorem:psin} without the completeness implies 
(ii) of Theorem~\ref{theorem:PSIN} and 
the orthogonality of $\{\Psi_n\}_{n=0}^\infty$ in $L^2(\mathbb{R})$. 
If (i) of Theorem~\ref{theorem:PSIN} holds, 
then the completeness of $\{\Psi_n\}_{n=0}^\infty$ in $L^2(\mathbb{R})$ 
follows immediately. 
Indeed. combining (i) of Theorem~\ref{theorem:PSIN} 
and the results in Section~\ref{section:modified} with 
$$
A=\frac{i(1+\beta^2)}{2\alpha^2},
\quad
B=-i,
\quad
C=\frac{(\alpha^2+\beta^2+1)\beta+i\alpha^2}{1+\beta^2},
$$
we can check the completeness of $\{\Psi_n\}_{n=0}^\infty$ in $L^2(\mathbb{R})$. 
This implies the completeness of 
$\{\psi_n\}_{n=0}^\infty$ in $\mathscr{H}_B(\mathbb{C})$ 
stated in Theorem~\ref{theorem:psin} 
since $B_1$ is a Hilbert space isomorphism of 
$L^2(\mathbb{R})$ onto $\mathscr{H}_B(\mathbb{C})$. 
For this reason, we have only to show (i) of Theorem~\ref{theorem:PSIN}. 
For this purpose, we need the following.
\begin{lemma}
\label{theorem:gauss}
Let $\rho>0$ and let $2\theta\in(-\pi/2,\pi/2)$. Then we have 
$$
\int_{\mathbb{R}}
\exp\Bigl(-\rho^2e^{2i\theta}t^2\Bigr)
dt
=
\frac{\sqrt{\pi}}{\rho e^{i\theta}}.
$$
\end{lemma}
\begin{proof}
The integrand is an even function of $t\in\mathbb{R}$. 
By using change of variable $t \mapsto \rho t$, we have 
$$
\int_{\mathbb{R}}
\exp\Bigl(-\rho^2e^{2i\theta}t^2\Bigr)
dt
=
\frac{2}{\rho}
\int_0^\infty
\exp\Bigl(-e^{2i\theta}t^2\Bigr)
dt. 
$$
Let $R>0$. Consider a contour $\gamma_R$ which consists of 
$\gamma_R=\gamma_R(1)\cup\gamma_R(2)\cup\gamma_R(3)$, where 
$$
\gamma_R(1)
=
\{t \ \vert \ t\in[0,R]\},
\quad
\gamma_R(3)
=
\{e^{i\theta}(R-t) \ \vert \ t\in[0,R]\},
$$
$$
\gamma_R(2)
=
\{Re^{it} \ \vert \ t\in[0,\theta]\}
\quad
(\theta\geqslant0),
\quad
\gamma_R(2)
=
\{Re^{i(\theta-t)} \ \vert \ t\in[\theta,0]\}
\quad
(\theta<0).
$$
Applying Cauchy's theorem to the holomorphic function $e^{-z^2}$ on $\gamma_R$, 
we have 
$$
e^{i\theta}
\int_0^R
\exp\Bigl(-e^{2i\theta}t^2\Bigr)
dt
=
\int_0^R
e^{-t^2}
dt
+
iR
\int_0^\theta
\exp\Bigl(-R^2e^{2it}+it\Bigr)
dt.
$$
Here we note that $0<\cos(2\theta)\leqslant1$ since $2\theta\in(-\pi/2,\pi/2)$. 
Then we deduce that 
\begin{align*}
  \Bigl\vert
  \exp\Bigl(-R^2e^{2it}+it\Bigr)
  \Bigr\vert
& =
  \exp\Bigl(-R^2\cos(2t)\Bigr)
  \leqslant
  \exp\Bigl(-R^2\cos(2\theta)\Bigr)
\\
& =
  \cfrac{1}{\displaystyle\sum_{k=0}^\infty\cfrac{\bigl(R^2\cos(2\theta)\bigr)^k}{k!}}
  \leqslant
  \frac{1}{R^2\cos(2\theta)},  
\end{align*}
and 
\begin{align*}
  \left\lvert
  iR
  \int_0^\theta
  \exp\Bigl(-R^2e^{2it}+it\Bigr)
  dt
  \right\rvert 
& \leqslant
  R
  \left\lvert
  \int_0^\theta
  \Bigl\lvert
  \exp\Bigl(-R^2e^{2it}+it\Bigr)
  \Bigr\rvert
  dt
  \right\rvert 
\\
& \leqslant
  R
  \left\lvert
  \int_0^\theta
  \frac{1}{R^2\cos(2\theta)}
  dt
  \right\rvert 
  =
  \frac{\lvert\theta\rvert}{R\cos(2\theta)}.
\end{align*}
Then we have 
$$
\int_0^R
\exp\Bigl(-e^{2i\theta}t^2\Bigr)
dt
=
\frac{1}{e^{i\theta}}
\int_0^R
e^{-t^2}
dt
+
\mathcal{O}\left(\frac{1}{R}\right)
\quad
(R \rightarrow \infty), 
$$
and 
$$
\int_0^\infty
\exp\Bigl(-e^{2i\theta}t^2\Bigr)
dt
=
\frac{1}{e^{i\theta}}
\int_0^\infty
e^{-t^2}
dt
=
\frac{\sqrt{\pi}}{2e^{i\theta}}. 
$$
This completes the proof.
\end{proof}
Finally we complete the proof of Theorem~\ref{theorem:PSIN}. 
\begin{proof}[{\bf Proof of Theorem~\ref{theorem:PSIN}}]
It suffices to show the part (i). 
We first compute the concrete form of $\Psi_0(x)$. 
Recall the definition of $\Psi_0(x)$  
$$
\Psi_0(x)
=
B_1^\ast\psi_0(x)
=
2^{-1/2}\pi^{-3/4}
\iint_{\mathbb{R}^2}
e^{G(y,\eta,x)}
dyd\eta,
$$
where
$$
G(y,\eta,x)
=
-
\frac{(y+i\eta)^2}{4}
+
(y+i\eta)x
-
\frac{x^2}{2}
-
\frac{a(y-i\eta)^2}{4}
-
\frac{y^2+\eta^2}{2}.
$$
Elementary computation gives
\begin{align*}
  G(y,\eta,x)
  =
& -
  \frac{\alpha^2+i(\alpha^2+\beta^2+1)\beta}{2(1+\beta^2)}
  x^2
\\
& -
  \frac{2\alpha^2+2\beta^2+1+i\beta}{2(\alpha^2+\beta^2+1)}
  (y-z_1)^2
  -
  \frac{1-i\beta}{2\alpha^2+2\beta^2+1+i\beta}
  (\eta-z_2)^2, 
\end{align*}
where 
$$
z_1
=
\frac{2x-i(1-a)\eta}{3+a},
\quad
z_2
=
i\frac{\alpha^2+\beta^2+i\beta}{1-i\beta}x.
$$
By using this and Lemma~\ref{theorem:gauss}, we deduce that 
\begin{align*}
  \Psi_0(x)
& =
  2^{-1/2}\pi^{-3/4}
  \exp\left(
      -
      \frac{\alpha^2+i(\alpha^2+\beta^2+1)\beta}{2(1+\beta^2)}
      x^2
      \right)
\\
& \times
  \int_{\mathbb{R}}
  \exp\left(
      -
      \frac{1-i\beta}{2\alpha^2+2\beta^2+1+i\beta}
      (\eta-z_2)^2
      \right)
  d\eta
\\
& \times
  \int_{\mathbb{R}}
  \exp\left(
      -
      \frac{2\alpha^2+2\beta^2+1+i\beta}{2(\alpha^2+\beta^2+1)}
      (y-z_1)^2
      \right)
  dy
\\
& =
  2^{-1/2}\pi^{-3/4}
  \exp\left(
      -
      \frac{\alpha^2+i(\alpha^2+\beta^2+1)\beta}{2(1+\beta^2)}
      x^2
      \right)
\\
& \times
  \int_{\mathbb{R}}
  \exp\left(
      -
      \frac{1-i\beta}{2\alpha^2+2\beta^2+1+i\beta}
      \eta^2
      \right)
  d\eta
\\
& \times
  \int_{\mathbb{R}}
  \exp\left(
      -
      \frac{2\alpha^2+2\beta^2+1+i\beta}{2(\alpha^2+\beta^2+1)}
      y^2
      \right)
  dy
\\
& =
  2^{-1/2}\pi^{-3/4}
  \times
  \left\{
  \frac{1-i\beta}{2\alpha^2+2\beta^2+1+i\beta}
  \cdot
  \frac{2\alpha^2+2\beta^2+1+i\beta}{2(\alpha^2+\beta^2+1)}
  \right\}^{-1/2}
  \pi
\\
& \times
  \exp\left(
      -
      \frac{\alpha^2+i(\alpha^2+\beta^2+1)\beta}{2(1+\beta^2)}
      x^2
      \right)
\\
& =
  A_{\alpha,\beta}
  \exp\left(
      -
      \frac{\alpha^2+i(\alpha^2+\beta^2+1)\beta}{2(1+\beta^2)}
      x^2
      \right),
\end{align*}
which is desired. 
\par
Finally we check the Rodrigues formula for $\Psi_n$. 
By using the definition of $\Psi_n$, we deduce that 
\begin{align*}
  \Psi_n(x)
& =
  B^\ast(\Lambda_{\alpha,\beta}^\ast)^n\psi_0(x)
  =
  (C_{\alpha,\beta}P_{\alpha,\beta}^\ast)^nB^\ast\psi_0(x)
  =
  (C_{\alpha,\beta})^n(P_{\alpha,\beta}^\ast)^n\Psi_0(x)
\\
& =
  (-C_{\alpha,\beta})^n
  \left(
  \frac{d}{dx}
  -
  \frac{\alpha^2-i(\alpha^2+\beta^2+1)\beta}{1+\beta^2}x
  \right)^n
  \Psi_0(x)
\\
& =
  (-C_{\alpha,\beta})^n
  \exp\left(
      \frac{\alpha^2-i(\alpha^2+\beta^2+1)\beta}{2(1+\beta^2)}x
      \right)
\\
& \times
  \left(\frac{d}{dx}\right)^n
  \left\{
  \exp\left(
      -
      \frac{\alpha^2-i(\alpha^2+\beta^2+1)\beta}{2(1+\beta^2)}x
      \right)
  \Psi_n(x)
  \right\}
\\
& =
  A_{\alpha,\beta}(-C_{\alpha,\beta})^n
  \exp\left(
      \frac{\alpha^2-i(\alpha^2+\beta^2+1)\beta}{2(1+\beta^2)}x
      \right)
  \left(\frac{d}{dx}\right)^n
  \exp\left(
      -
      \frac{\alpha^2}{1+\beta^2}x
      \right).
\end{align*}
This completes the proof.
\end{proof}
%
%
%%%%%%
%%%%%% bibliography
%%%%%%

%%%%%%
%%%%%% End
%%%%%%
\end{document}